\documentclass{amsart}

\usepackage{amssymb, amsmath}
\usepackage{amsthm}
\usepackage{mathrsfs}
\usepackage{amscd}
\usepackage{hyperref}

\theoremstyle{plain}
\newtheorem{theorem}{Theorem}
\theoremstyle{remark}

\newtheorem{example}[theorem]{Example}
\theoremstyle{plain}
\newtheorem{corollary}[theorem]{Corollary}
\newtheorem{lemma}[theorem]{Lemma}
\newtheorem{proposition}[theorem]{Proposition}

\def\R{{\mathbb R}}

\renewcommand{\P}{{\mathbb P}}

\newcommand{\Om}{\Omega}

\renewcommand{\d}{\delta}
\newcommand{\e}{\varepsilon}

\renewcommand{\o}{\omega}
\renewcommand{\O}{\Omega}

\newcommand{\beq}{\begin{equation}}
\newcommand{\eeq}{\end{equation}}
\newcommand{\bal}{\begin{aligned}}
\newcommand{\eal}{\end{aligned}}
\newcommand{\ben}{\begin{enumerate}}
\newcommand{\een}{\end{enumerate}}
\newcommand{\bit}{\begin{itemize}}
\newcommand{\eit}{\end{itemize}}

\newcommand{\bth}{\begin{theorem}}
\renewcommand{\eth}{\end{theorem}}
\newcommand{\bpr}{\begin{proposition}}
\newcommand{\epr}{\end{proposition}}
\newcommand{\ble}{\begin{lemma}}
\newcommand{\ele}{\end{lemma}}
\newcommand{\bpf}{\begin{proof}}
\newcommand{\epf}{\end{proof}}
\newcommand{\bex}{\begin{example}}
\newcommand{\eex}{\end{example}}
\newcommand{\bre}{\begin{example}}
\newcommand{\ere}{\end{example}}

\newcommand {\idual}[3] {\langle #1, #2 \rangle_{#3}}
\newcommand {\dual}[2] {\idual{#1}{#2}{} }

\newcommand{\D}{{\mathcal D}}

\newcommand{\n}{\Vert}
\newcommand{\one}{{{\bf 1}}}
\newcommand{\embed}{\hookrightarrow}
\newcommand{\s}{^*}
\newcommand{\lb}{\langle}
\newcommand{\rb}{\rangle}

\newcommand{\limn}{\lim_{n\to\infty}}
\newcommand{\limk}{\lim_{k\to\infty}}
\newcommand{\limj}{\lim_{j\to\infty}}
\newcommand{\sumn}{\sum_{n=1}^\infty}
\newcommand{\sumk}{\sum_{k=1}^\infty}

\begin{document}

\title[Compactness in vector-valued Banach function spaces]
{Compactness in vector-valued \\ Banach function spaces}
\author{Jan van Neerven}

\address{Delft Institute of Applied Mathematics\\
Technical University of Delft \\ P.O. Box 5031\\ 2600 GA Delft\\The
Netherlands}

\email{J.vanNeerven@math.tudelft.nl}

\thanks{The author is supported by the `VIDI subsidie' 639.032.201
in the `Vernieuwingsimpuls' programme of the Netherlands
Organization for Scientific Research (NWO) and
by the Research Training Network HPRN-CT-2002-00281.}



\begin{abstract}
We give a new proof of a recent characterization by Diaz and
Mayoral of compactness in the Lebesgue-Bochner spaces $L_X^p$, 
where $X$ is a Banach space and $1\le p<\infty$,
and extend the result to vector-valued Banach function spaces
$E_X$, where $E$ is a Banach function space with order continuous norm.
\end{abstract}

\maketitle

Let 
$X$ be a Banach space.
The problem of describing the compact sets in the Lebesgue-Bochner spaces $L_X^p$,
$1\le p<\infty$, goes back to the work of Riesz, Fr\'echet,  Vitali 
in the scalar-valued case, cf. \cite{DS}, and 
has been considered by many authors, cf. \cite{Au,Br,DM,RS,Si}. 
In a recent paper, Diaz and Mayoral \cite{DM} 
proved that if the underlying  measure space is finite, then
a subset $K$ of $L_X^p$ is relatively compact 
if and only if $K$ is uniformly $p$-integrable, 
scalarly relatively compact, and either uniformly tight or 
flatly concentrated. Their proof relies on the 
Diestel-Ruess-Schachermayer characterization \cite{DRS} of weak compactness
in $L_X^1$ and the notion of Bocce oscillation, 
which was studied recently by Girardi \cite{G} and 
Balder-Girardi-Jalby \cite{BGJ} in the context of compactness in $L_X^1$. 
The purpose of this note is to present 
an extension of the Diaz-Mayoral result to 
vector-valued Banach function spaces $E_X$,
with a proof based on Prohorov's tightness theorem.

\medskip
We begin with some preliminaries on
Banach lattices and Banach function spaces. Our terminology is standard and follows \cite{MN}.

A Banach lattice $E$ is said to have {\em order continuous norm} if
every net in $E$ which decreases to $0$ converges to $0$. 
Every separable Banach function space $E$ has this property. Indeed, because 
such spaces are Dedekind complete \cite[Lemma 2.6.1]{MN}
and cannot contain an isomorphic copy of $l^\infty$, this follows from 
\cite[Corollary 2.4.3]{MN}. 

A subset $F$  of a Banach lattice 
$E$ is called {\em almost order bounded} if for every $\e>0$ there exists an element $x_\e\in E_+$ such that
$F\subseteq [-x_\e,x_\e]+ B(\e)$, where 
$[-x_\e,x_\e] := \{y\in E: \ -x_\e\le y\le x_\e\}$ and $B(\e) :=\{x\in X: \ \n
x\n < \e\}$. It follows from  \cite[Theorem 2.4.2]{MN} that
every almost order bounded set in a Banach lattice with order continuous norm
is relatively weakly compact. 

\begin{lemma}\label{lem:bdd-interval} Let $E$ be a Banach lattice and let $I$ be a dense ideal in $E$. 
If the set $A\subseteq E^+$ is almost order bounded, then for every $\e>0$ there exists an element
$x_\e\in I^+$ such that $A\subseteq [0,x_\e] + B(\e)$.
\end{lemma}

\begin{proof} Fix $\e>0$ and choose $y_\e\in E^+$ such that
$A\subseteq [-y_\e,y_\e] + B(\frac12\e)$. 
Choose $x_\e\in I$
such that $0\le x_\e\le y_\e$ and $\n y_\e -x_\e\n <\frac12\e$. 

Fix $a\in A$, say $a = y+b$ with $y\in [-y_\e,y_\e]$ and $\n b\n<\frac12\e$.  
With $z_\e: = y_\e +|b|$ we have $\n z_\e -x_\e\n \le \n y_\e -x_\e\n+ \n b\n < \e$.
From $a\le z_\e$ we infer
$(a-x_\e)^+ \le (z_\e-x_\e)^+ = z_\e - x_\e$ and hence
$ \n (a-x_\e)^+\n \le \n z_\e-x_\e\n < \e$. It follows that
$a= a \wedge x_\e + (a - x_\e)^+ \in [0,x_\e] + B(\e)$.
\end{proof}

If $E$ is a Banach function space with order continuous norm,
then for all $f\in E$ we have
$
\lim_{r\to\infty} \n \one_{\{|\phi| > r\}} \phi\n_{E} = 0.
$
Motivated by this we shall call a subset $F$ of $E$ {\em uniformly $E$--integrable} if 
$$
\lim_{r\to\infty} \sup_{\phi\in F} \big\n \one_{\{|\phi|> r\}} {\phi}\big\n_{E} 
=0.
$$
For $E=L^p$ with $1\le p<\infty$, this definition reduces to the 
classical definition of uniform $p$-integrability.

If $E$ is a Banach function space containing the constant function $\one$,
then every uniformly $E$--integrable subset of $E$ is almost order bounded.
From Lemma \ref{lem:bdd-interval} we deduce the following converse:

\begin{lemma}\label{lem:ui}
Let $E$ be a Banach function space with order continuous norm
over a $\sigma$-finite measure space $(S,\nu)$. If  $F\subseteq E^+$
is almost order bounded, then $F$ is uniformly $E$-integrable.
\end{lemma}
\begin{proof}
Let $\e>0$ be fixed. By Lemma \ref{lem:bdd-interval}, applied to 
$I:= E\cap L^\infty(S,\nu)$, we may choose $x_\e \in E^+$ and real numbers $R_\e\ge 0$ 
such that $0\le x_\e \le R_\e$ $\nu$-almost everywhere and
$F \subseteq [0,x_\e]+ B(\e)$. Keeping $\phi\in F$ fixed for the moment, 
 we can write $\phi = x+b$ with $x\in [0,x_\e]$ 
and $\n b\n_E<\e$.
Then, for all $r>0$,
$$ 
\bal
\big\n \one_{\{\phi > r\}}\phi \big\n_{E}
& \le \big\n \one_{\{\phi>r\}} x \big\n_E 
+ \big\n \one_{\{\phi>r\}} b\big\n_E
\\ & \le \big\n \one_{\{x>\tfrac12 r\}} x \big\n_E + \big\n
\one_{\{|b|>\tfrac12 r\}} x\big\n_E +  \n  b\n_E
\\ & \le \big\n \one_{\{x_\e>\tfrac12 r\}} x_\e \big\n_E
 + \frac{2R_\e}{r} \big\n b \big\n_E + \e,
\eal
$$
where in the last step we used that $\nu$-almost everywhere we have
$$ 0\le \tfrac12 r \one_{\{|b|>\tfrac12 r\}} x\le |b|x \le |b|x_\e \le R_\e
|b|.$$ 
The lemma immediately follows from this. 
\end{proof}

The next lemma gives a
sufficient condition for norm convergence in almost order bounded sets.
Recall that an element $x\s\in E\s$ in the dual of a Banach lattice $E$ is
called {\em stricly positive} if $\lb |x|,x\s\rb = 0$ implies $x=0$.

\begin{lemma}\label{lem:alm-ord-bdd}
Let $E$ be a Banach lattice with order continuous norm and 
let $F$ be an almost order bounded subset of $E$. If $(x_{j})_{j\ge 1}$
is a sequence in $F$ such that 
$ \lim_{j\to\infty} \lb |x_{j}|, x\s\rb = 0$
for some strictly positive element $x\s\in E\s$, then
$\lim_{j\to\infty} x_{j} = 0$ in $E$.
\end{lemma}
\begin{proof}
Assume the contrary and choose sequences $j_n\to\infty$
and a number $\d>0$ such that $\n x_{j_n}\n_E \ge \d$ for all $n$.
We have
$$ \lim_{m,n\to\infty} \lb  |x_{j_m} - x_{j_n}|, x\s\rb
\le  \lim_{m\to\infty}\lb  |x_{j_m}|,x\s\rb  + \lim_{n\to\infty}\lb  |x_{j_n}|,x\s\rb =0$$
and therefore, by \cite[Lemma 3.8]{Rab}, $\lim_{n\to\infty} x_{j_n} =x$ for some $x\in E$.
Then $\n x\n \ge \d$ and 
$ 0= \lim_{n\to\infty} \lb |x_{j_n}|,x\s\rb = \lb |x|,x\s\rb.$
This contradicts the fact that $x\s$ is strictly positive. 
\end{proof}

Let  $X$ be a Banach space. 
A set $M$ of Radon probability measures 
on $X$ is called {\em uniformly tight}  if for every $\e>0$ there exists
a compact set $K$ in $X$ such that $$\mu(K)\ge 1-\e \quad\forall \mu\in M.$$ 
By Prohorov's theorem for Radon measures \cite[Theorem I.3.6]{VTC}, $M$ is uniformly tight if and only if $M$ relatively weakly compact, i.e., 
every sequence $(\mu_n)_{n \ge 1}$ has a subsequence $(\mu_{n_k})_{k \ge 1}$ such that for some Radon probability measure $\mu$ we have
$$ \limk \int_X f\,d\mu_{n_k} = \int_X f\,d\mu \ \ \hbox{for all $f\in C_b(X)$},$$
where $C_b(X)$ is the space of all scalar-valued bounded continuous functions on $X$.

\medskip
We shall formulate the main result of this paper 
for Banach function spaces $E$ over a probability space $(\O,\P)$. 
This is done merely for convenience; 
the result extends to arbitrary finite measure spaces by a trivial normalization
argument.

The space $E_X$ of all strongly $\P$-measurable
functions $\phi: \O\to X$ such that $\o\mapsto \n \phi(\o)\n$ belongs to $E$
is a Banach space with respect to the norm
$$
\n \phi\n_{E_X} := \big\n \,\n \phi \n\,\big\n_E.
$$
Here, as usual, we identify functions that are equal $\P$-almost everywhere.
It follows from \cite[Proposition 2.6.3]{MN} that $\limn \phi_n = \phi$ in $E_X$ implies that for some subsequence
we have $\limk \phi_{n_k}(\o) = \phi(\o)$ in $X$ for $\P$-almost all $\o\in \O$.

The {\em distribution} of a function $\phi\in E_X$
 is the Radon probability measure $\mu_{\phi}$ on $X$ defined by
$$ \mu_\phi(B) = \P\{\phi\in B\}  
\ \ \hbox{for $B\subseteq X$ Borel.}$$
This definition is independent of the representative of $\phi$ 
used to define $\mu_\phi$. 



We call a subset $F$ of $E_X$:
\bit
\item {\em almost order bounded}, if $\{\n \phi\n: \ \phi\in F\}$ is almost order bounded in $E$;
\item {\em scalarly relatively compact}, if $\{\lb \phi,x\s\rb: \ \phi\in F\}$ is relatively norm compact in $E$ for all $x\s\in E\s$;
\item {\em uniformly tight}, if $\{\mu_\phi: \ \phi\in F\}$
is uniformly tight.
\eit

\begin{lemma}\label{lem:obdd}
 Let $F$ be a subset of $E_X$.
If $F$ is almost order bounded,
then also $F-F$ is almost order bounded.
\end{lemma}
\begin{proof}
Fix $\e>0$. Using Lemma \ref{lem:bdd-interval} we choose $x_\e\in E^+$ such that $\n\phi\n \in [0,x_\e]+B(\frac12\e)$ for all $\phi\in F$. 

{\em Step 1} -- We claim that each $\phi\in F$ can be written as $\phi = f+g$ with $\n f\n\in [0,x_\e]$ and $g\in B(\frac12\e)$.
Indeed, we have 
$$
\phi 
 = \Big(\one_{\{\n \phi\n\le x_\e\}}\phi + \one_{\{\n \phi\n > x_\e\}}\frac{x_\e}{\n \phi\n} \phi\Big) + \one_{\{\n \phi\n > x_\e\}}\frac{(\n \phi\n - x_\e)}{\n \phi\n} \phi .
$$
For the first term on the right hand side 
we have
$$\Big\n\one_{\{\n \phi\n\le x_\e\}}\phi + \one_{\{\n \phi\n > x_\e\}}\frac{x_\e}{\n \phi\n} \phi\Big\n \in   [0,x_\e].$$
Writing $\n \phi\n = a+b$ with $a\in [0,x_\e]$ and $\n b\n_E<\frac12\e$, for the second term we have
$$  \Big\n \one_{\{\n \phi\n > x_\e\}}\frac{(\n \phi\n - x_\e)}{\n\phi\n} \phi\Big\n = \one_{\{\n \phi\n > x_\e\}} (a+b - x_\e)
\le \one_{\{\n \phi\n > x_\e\}} b,
$$
which 
shows that 
$$ \Big\n \one_{\{\n \phi\n > x_\e\}}\frac{(\n \phi\n - x_\e)}{\n\phi\n}\phi\Big\n_{E_X} \le \n b\n_E < \tfrac12\e.
$$
This proves the claim.

{\em Step 2} -- 
Let $\phi_1,\phi_2\in F$ be given, and write
$\phi_k = f_k+g_k$, where $\n f_k\n\in [0,x_\e]$ and $g_k\in B(\frac12\e)$ for $k=1,2$. Then
$$ 
\n \phi_1-\phi_2\n = \n f_1-f_2\n + \big(\,\n \phi_1-\phi_2\n - \n f_1-f_2\n\,\big),
$$
with 
$\n f_1-f_2\n \in [0,2x_\e] $
and
$$\big|\,\n \phi_1-\phi_2\n - \n f_1-f_2\n\,\big|
\le \n g_1-g_2\n,
$$
which shows that $\big\n \,\n\phi_1-\phi_2\n - \n f_1-f_2\n\,\big\n_E
< \e$.
\end{proof}

\begin{theorem}\label{thm:DM} 
Let ${E}$ be a Banach function
space with order continuous norm over a probability space $(\O,\P)$.
Let $X$ a Banach space. 
For a subset $F$ of $E_X$ the following assertions are equivalent:
\begin{enumerate}
\item\label{item:DM-a} 
    The set $F$ is relatively compact;
\item\label{item:DM-b} 
    The set $F$ is 
    uniformly tight, almost order bounded, and
    scalarly relatively compact.
\end{enumerate}\end{theorem}

As has been mentioned above, every separable Banach function space has order continuous norm.

\begin{proof}
Without loss of generality we may assume that $E$ is {\em saturated}, i.e., 
that $f\equiv 0$ on $A$ for all $f\in E$ implies $\P(A)=0$ \cite[Section 67]{Za}.  

(\ref{item:DM-a})$\Rightarrow$(\ref{item:DM-b}): \   
It is clear that the relative compactness of $F$ implies its almost order boundedness  and scalar relative compactness. 


To prove the uniform tightness of $F$,
by Prohorov's theorem it suffices to show that every sequence 
$(\phi_n)_{n\ge 1}$ in $F$ has a subsequence $(\phi_{n_j})_{j\ge 1}$ 
whose distributions converge weakly.

Let us write $\mu_n :=
\mu_{\phi_n}$ for simplicity. 
Since $F$ is compact we may assume, by passing to a subsequence, that $(\phi_n)_{n\ge 1}$ converges in $E_{X}$ to an element
$\phi\in E_{X}$. By passing to a further subsequence we may
also assume that the convergence takes place almost surely.  
Let $\mu: =\mu_\phi$ be the distribution of $\phi$. 
Then for all $f\in C_b(X)$ we have, by dominated convergence, 
$$
\limn \int_{X} f\, d\mu_n 
= \limn \int_\O f\circ \phi_n \, d\P 
= \int_\O f\circ \phi \, d\P = \int_{X} f\, d\mu.
$$ 

(\ref{item:DM-b})$\Rightarrow$(\ref{item:DM-a}): \
Let $(\phi_n)_{n\ge 1}$ be a sequence in $F$. We shall prove that
some subsequence $(\phi_{n_j})_{j\ge 1}$ converges in $E_X$.

{\em Step 1} -- Let $\nu_{n,m}$ denote distribution of the random
variable $\phi_{n}-\phi_{m}$. We claim that the family $(\nu_{n,m})_{n,m\ge 1}$ 
is uniformly tight. The proof is standard and runs as follows. 
Fix some $\e>0$. Since
$(\mu_n)_{n\ge 1}$ is uniformly tight we may choose a compact
set $K\subseteq X$ such that $\mu_n(K)\ge 1-\e$ for all $n\ge 1$. The
set  $L = \{x-y:\ x,y\in K\}$ is compact as well, being the image of
the compact set $K\times K$ under the continuous map 
$(x,y)\mapsto x-y$. Noting that  $\phi_{n}(\o)\in K$ and
$\phi_{m}(\o)\in K$ implies  $\phi_{n}(\o)-\phi_m(\o)\in L$, 
the claim now follows from
$$\bal
\nu_{n,m}(L) 
& \ge \P\{\phi_{n}\in K, \ \phi_{m}\in K\}\\ 
&\ge 1 - \bigl(\P\{\phi_{n}\in \complement K\} + \P\{\phi_{m}\in \complement K\} \bigr)
  =  1 - \bigl(\mu_{n}(K) + \mu_{m}(K)\bigr) \ge 1-2\e.
\eal $$

{\em Step 2} -- 
Since $F$ is uniformly tight, we may assume $X$ to
be separable. Let $(x_m\s)_{m\ge 1}$ 
be a sequence in $X\s$ whose intersection with every ball is weak$\s$-dense.
As before we let $\mu_n$ denote the distribution of $\phi_n$. 
Prohorov's theorem implies the existence of a weakly convergent
subsequence $(\mu_{n_j})_{j\ge 1}$.  By passing to a subsequence we may assume
that the limit $\psi_m:=\limj \lb \phi_{n_j},x_m\s\rb$ exists in ${E}$ 
for all $m$ and that the convergence happens almost surely. 

We claim that $\nu_{n_j,n_k} \to
\delta_0$ weakly as $j,k \to \infty$, where $\delta_0$ denotes the Dirac measure concentrated at $0$. 
Let $j_l\to \infty$ and $k_l\to\infty$.
By Step 1 we may pass to a subsequence of the indices $l$
and assume that $\nu_{n_{j_l},n_{k_l}}\to \nu$ for some Radon probability
measure $\nu$ on $X$.
By taking Fourier transforms,
from the almost sure convergence $\lim_{l\to\infty} \lb \phi_{n_{j_l}},x_m\s\rb 
= \lim_{l\to\infty} \lb \phi_{n_{k_l}},x_m\s\rb = \psi_m$
we see that for all $m$,
$$\widehat{\nu}(x_m\s) = \lim_{l\to\infty}
\widehat{\nu_{n_{j_l},n_{k_l}}}(x_m\s) = 
\lim_{l\to\infty}\int_\Om 
\exp(-i \dual{\phi_{n_{j_l}}-\phi_{n_{k_l}}}{x_m^*}) d\P 
= 1 = \widehat{\delta_0}(x_m\s)
$$
by dominated convergence.
Noting that the weak$\s$-topology of every ball in $X\s$ is metrizable,
combined with the fact that the Fourier transforms of Radon 
probability measures
are weak$\s$-sequentially continuous,
it follows that $\widehat{\nu} = \widehat{\delta_0}$. Therefore
$\nu=\d_0$ by the uniqueness of the Fourier transform. 
Since the sequences $j_l$ and $k_l$ were arbitrary, this proves the
claim.

{\em Step 3} -- 
It remains to show that the sequence $(\phi_{n_j})_{j\ge 1}$ is Cauchy in $E_X$.

For $j,k\ge 1$ define the functions $g_{jk}\in E$ by 
$$g_{jk}:= \n \phi_{n_j}- \phi_{n_k}\n.$$
For $n\ge 1$ choose $r_n\ge 0$ so large that
$$ \n \one_{\{g_{jk} > r_n\}}g_{jk}\n_E < \tfrac1n \ \ \hbox{for all $j,k\ge 1$.}$$
This is possible since $F-F$ is almost order bounded by Lemma \ref{lem:obdd}. 
By Lemma \ref{lem:ui}, $\n F-F\n$ is uniformly $E$-integrable.

Let $f\in C_b(\R)$ be arbitrary. By Step 2 and Prohorov's theorem,
$$\lim_{j,k\to\infty} \int_{\O} f\circ g_{jk}\,d\P =f(0).$$
Keeping $n\ge 1$ fixed for the moment and taking $f(t) = |t|\wedge r_n$, it follows that there exists an index $N_n\ge 1$ such that
$$ \int_{\O} g_{jk}\wedge r_n \,d\P < \tfrac1n \ \ \hbox{for all $j,k\ge N_n$.}$$
Let $0\le \psi_0\le \one$ be a  $\P$-almost everywhere strictly positive 
function belonging to the associate space $E'$, which is defined as the space  of all
$\nu$-measurable functions $\psi$ on $S$ such that
$$ \n \psi\n_{E'} := \sup_{\n \phi\n_E \le 1} \int_\O |\phi \psi|\,d\P < \infty.$$
Such a function exists since $E$ is assumed to be saturated.
Note that $\psi_0$ 
is strictly positive as element of $E\s$.
For $j,k\ge N_n$,
$$
\bal 
0\le \lb g_{jk},\psi_0\rb 
&\le \lb g_{jk}\wedge r_n,\psi_0\rb + \lb \one_{\{g_{jk} > r_n\}} g_{jk},\psi_0\rb
< \tfrac1n (1+ \n \psi_0\n_{E'}).
\eal
$$
It follows that 
$\lim_{j,k\to\infty} \lb g_{jk},\psi_0 \rb = 0$.
Now Lemma \ref{lem:alm-ord-bdd} shows that
$\lim_{j,k\to\infty} g_{jk}=0$ in $E$.
\end{proof}

As in \cite{DM}, the uniform tightness assumption in assertion (2) 
may be replaced by flat concentration. This follows from Prohorov's 
theorem in combination with the well known result of 
de Acosta \cite{Ac}, see also \cite[Theorem I.3.7]{VTC}, 
that a family $M$ of Radon probability measures on $X$ is 
uniformly tight if and only if $M$ is flatly concentrated and 
for all $x\s\in E\s$ the set of image measures
$\lb M,x\s\rb = \{\lb \mu,x\s\rb: \ x\s\in E\s\}$ is uniformly tight.

\end{document}